\documentclass[12pt]{amsart}
\usepackage[utf8]{inputenc}
\usepackage[margin=1.25in]{geometry}
\usepackage{hyperref}
\usepackage{amssymb}
\usepackage{tikz-cd}
\usepackage{enumitem}
\usepackage{soul}
\input xy
\usepackage[all]{xy}
\usepackage{tikz-cd}

\usepackage{placeins}

\newcommand{\can}{\mathbf{can}}

\newcommand{\one}{\mathbf{1}}

\renewcommand{\Vec}{\text{Vec}}

\newcommand{\cC}{{\mathcal C}}

\newcommand{\ot}{{\otimes}}

\newcommand{\TY}{{\mathcal{TY}}}

\newcommand{\ku}{{\Bbbk}}
\newcommand{\Z}{{\mathbb Z}}

\newcommand{\C}{{\mathcal{C} }}

\newcommand{\id}{\operatorname{id}}

\newcommand{\cB}{\mathcal{B}}
\newcommand{\cD}{\mathcal{D}}

\newcommand{\Obj}{\mbox{\rm Obj\,}}

\newcommand{\Hom}{\operatorname{Hom}}

\newcommand{\End}{\operatorname{End}}

\newcommand{\Id}{\operatorname{Id}}

\newcommand{\Aut}[1]{\operatorname{Aut}_\otimes(#1)}

\theoremstyle{plain}

\numberwithin{equation}{section}

\newtheorem{theorem}{Theorem}[section]
\newtheorem{lemma}[theorem]{Lemma}

\newtheorem{corollary}[theorem]{Corollary}
\newtheorem{proposition}[theorem]{Proposition}

\theoremstyle{definition}
\newtheorem{definition}[theorem]{Definition}
\newtheorem{example}[theorem]{Example}

\theoremstyle{remark}

\newtheorem{remark}[theorem]{Remark}

\theoremstyle{remark}

\newcounter{commentcounter}

\newcounter{todocounter}

\author[C\'esar Galindo]{C\'esar Galindo}
\address{ Departamento de Matem\'aticas, Universidad de los Andes, Bogot\'a, Colombia}
\email{cn.galindo1116@uniandes.edu.co}
\begin{document}

\title[Trivializing group actions]{Trivializing group actions on braided crossed tensor categories and graded braided tensor categories}

\thanks{The author would like to thank the hospitality and excellent working conditions of the Department of Mathematics at the University of Hamburg, where he carried out this research as a Fellow of the Humboldt Foundation.}
\begin{abstract}
For an abelian group $ A $, we study a close connection between braided crossed $ A $-categories with a trivialization of the $ A $-action and $ A $-graded braided tensor categories. Additionally,  we prove that the obstruction to the existence of a trivialization of a categorical group action $T$ on a monoidal category $\cC$ is given by an element  $O(T)\in H^2(G,\Aut{\Id_{\cC}})$. In the case that $O(T)=0$, the set of obstructions form a torsor over $\operatorname{Hom}(G,\Aut{\Id_{\cC}})$, where $\Aut{\Id_{\cC}}$ is the abelian group of tensor natural automorphisms of the identity.   

The cohomological interpretation of trivializations, together with the homotopical classification of (faithfully graded) braided $A$-crossed tensor categories developed in \cite{ENO3}, allows us to provide a method for the construction of faithfully $A$-graded braided tensor categories. We work out two examples. First, we compute the obstruction to the existence of trivializations for the braided crossed category associated with a pointed semisimple tensor category. In the second example, we compute explicit formulas for the braided $\Z/2$-crossed structures over Tambara-Yamagami fusion categories and, consequently, a conceptual interpretation of the results in \cite{siehler2000braided} about the classification of braidings over Tambara-Yamagami categories.
\end{abstract}


\date{\today}
\maketitle

 \section*{Introduction}

The notion of braided $G$-crossed tensor category introduced by Turaev in \cite{T}  has played an essential role in the recent application of fusion categories to enriched symmetries in condensed matter physics and the construction of Homotopical  TFTs, \cite{MR2959440, MR3195545,MR4021927,MR2183964,MR1923177,MR4033513,MR3555361}. 
 Recently in \cite{jones20203categorical}, higher categorical interpretations of braided $G$-crossed tensor categories have been developed, allowing a better understanding of the reason for its presence in different theories. 
 
In \cite{ENO3}, the authors studied braided $G$-crossed fusion categories using invertible module categories over braided fusion categories. They reduce the classification problem of braided $G$-crossed fusion categories with trivial component a braided fusion category $\cB$, to the classification (up to homotopy) of maps from $BG$ to  $B \operatorname{Pic}(\cB)$  (the classifying spaces of the monoidal 2-category of invertible $\cB$-module categories). This approach allows them to use the obstruction theory for homotopy classes of maps into the associated  Postnikov towers and provide an elegant and useful group cohomological parametrization of (faithfully graded) braided $G$-crossed fusion categories.

The aim of this note is to discuss, for any abelian group $ A $, a close connection between braided crossed $ A $-categories with a trivialization of the $ A $-action and $ A $-graded braided tensor categories. The existence of trivializations of categorical actions of groups and its classification has a straightforward cohomological interpretation (see Theorem \ref{obstruccion general}). The obstruction of the existence of a trivialization of a $G$-action $T$ on a tensor category $\cC$ is given by an element in $O(T)\in H^2(G,\Aut{\Id_{\cC}})$ and in case $O(T)=0$, the set of obstructions form a torsor over $\operatorname{Hom}(G,\Aut{\Id_{\cC}})$, where here $\Aut{\Id_{\cC}}$ means the abelian group tensor natural automorphisms of the identity.   The cohomological interpretation of trivialization, together with the homotopical classification of (faithfully graded) braided $A$-crossed tensor categories developed in \cite{ENO3}, allows us to provide a method for the construction of faithfully $A$-graded braided tensor categories. We consider two examples in the paper. The first one is the computation of the obstruction to the existence of trivializations for the braided crossed category associated with a pointed semisimple tensor category. As a second example, we compute explicit formulas for the braided $\Z/2$-crossed structures over Tambara-Yamagami fusion categories and, consequently, a conceptual interpretation of the results \cite{siehler2000braided} about the classification of braiding over Tambara-Yamagami categories.

Recently, Davydov and Nikshych in \cite{davydov2020braided}  proved that braided finite tensor categories (faithfully) graded by a finite group $A$ are in correspondence to braided monoidal 2-functors from $A$ to certain braided monoidal 2-category. In the spirit of \cite{ENO3}, in \emph{loc cit}, the obstruction and parametrization of these braided monoid 2-functors were developed using the Eilenberg-Mac ~ Lane cohomology. We hope our approach for constructing group-graded braided tensor categories can be considered as a complement to the methods developed in \cite{davydov2020braided}.

The paper's organization is as follows: In Section 1, we recall some basic definitions of groups' actions on monoidal categories. In Section 2, we discuss the obstruction and parametrization of the trivialization of group actions on tensor categories. In Section 3, we introduce the 2-category of braided $A$-crossed tensor categories with a trivialization and proved its equivalence with the 2-category of $A$-graded braided tensor categories. We worked out the example of semisimple pointed tensor categories. In section 4, we explicitly described formulas for the braided $\Z/2$-crossed structures on Tambara-Yamagami categories, and the case of Ising categories is presented in detail.
\section{Preliminaries}\label{prelim}

\subsection{Notation}

Let $\cC$ be a category. We denote by $\Obj(\cC)$ the class of objects of $\cC$ and by $\Hom_\cC(X, Y )$ the set of morphisms in $\cC$ from an object $X$ to an object $Y$. Also, by abuse of notation, $X \in \cC $ means that $X$ is an object of $\cC$.

The symbols  $\cC$ and $\cD$ will denote monoidal categories with unit objects $\one_\cC$ and $\one_\cD$ respectively. If no confusion arises, we will indicate the unit object of a monoidal category just by $\one$. To simplify computations and statements, by monoidal category, we will mean a strict monoidal category, and this is justified by  the  coherence theorem of S. MacLane. 

\subsection{Group actions on monoidal categories}\label{sectio defini G-category}
Let $G$ be a  group. We will denote by $\underline{G}$ the discrete monoidal category with $\Obj(\underline{G})=G$ and monoidal structure defined by the multiplication of $G$. If $\cC$ is a monoidal category, we will denote by $\underline{\Aut{\cC}}$ the monoidal category of monoidal autoequivalences of $\cC$ and natural monoidal isomorphism with tensor product given by the composition of monoidal functors.

An action of $G$ on $\cC$  is a monoidal functor $T:\underline{G}\to \underline{\Aut{\cC}}$. A $G$-action on $\cC$ defines the following data:
\begin{itemize}
\item monoidal functors $T(g):\C\to \C$ for each $g\in G$,
\item monoidal natural isomorphisms $T_2(g,h): T(gh)\to T(g)\circ T(h)$ for each pair $g,h\in G$,
\end{itemize}
such that 

\[
\xymatrix{
&T(g)\circ T(h)\circ T(k) \ar[dl]_{b(g,h)\circ \Id_{T(k)}} \ar[dr]^{\Id_{T(g)}\circ b(h,k)}& \\ 
T(gh)\circ T(k) \ar[dr]_{b(gh,k)} &&T(g)\circ T(hk) \ar[dl]^{b(g,hk)}\\
&T'(ghk)&
}
\]for all $g, h, k\in G$.
A $G$-action  is called \emph{strict} if $T(g)$ are strict monoidal and $T_2(g,h)$ are identities for all $g,h\in G$.




By  \cite[Theorem 1.1]{GALINDO-Coherence} every monoidal category with a $G$-action is canonically $G$-equivariant equivalent to a monoidal category with a strict $G$-action. Using \cite[Theorem 1.1]{GALINDO-Coherence}, we could assume without loss of generality that every $G$-action is strict.

\section{Trivializations of $G$-actions}

In this section, we define the notion of a trivialization of action of a group on a monoidal category and define an obstruction to the existence of trivializations.
\begin{definition}
Let $T:\underline{G}\to \Aut{\cC}$ be an action of group $G$ on a monoidal category $\cC$. 
A trivialization $\eta$ of $T$ consist of a family of \emph{monoidal} natural isomorphisms $\eta_g: T(g)\to \operatorname{Id}_{\cC}$, for all $g \in G$, such that 

\begin{align}
\eta_g\circ \eta_h=\eta_{gh}\circ T_2(g,h), && \forall g,h \in G.    
\end{align}

\end{definition}

\begin{remark}\label{rmk:trivial action}
A trivialization of an action $T:\underline{G}\to \Aut{\cC}$ is a just a monoidal functor from $T$ to the trivial action.
\end{remark}

\begin{lemma}\label{lemma: transport structure}
Let $F:\cC\to \cD$ be a monoidal functor and $H:\cC\to \cD$ a functor and natural isomorphism $\gamma:F\to H$. Hence there is a unique monoidal structure on $H$ such that that $\gamma$ is a monoidal isomorphism. Moreover, the monoidal structure on $H_2(X,Y):H(X\ot Y)\to H(X)\ot H(Y)$ is given by the commutativy of the diagram
\begin{equation}
    \xymatrix{
    H(X)\ot H(Y) \ar[rr]^{H_2(X,Y)}&& H(X\ot Y)\\
    F(X)\ot F(Y) \ar[u]^{\gamma_X\ot \gamma_Y} \ar[rr]^{F_2(X,Y)}&& F(X\ot Y) \ar[u]_{\gamma_{X\ot Y}}
    }
\end{equation}
\end{lemma}
\begin{proof}
It is a straightforward exercise on the transport of structures in category theory.
\end{proof}
\begin{theorem}\label{obstruccion general}
Let $T$ be an action of a group $G$ on a tensor category $\cC$ such that $\sigma_*\sim_\otimes \operatorname{Id}_{\cC}$ for all $\sigma\in G$.  Let $\chi(\sigma):\sigma_*\to \operatorname{Id}_\cC$ be monoidal natural isomorphisms for each $g\in G$. 

Define for every pair $g,h\in G$, a monoidal natural automorphism of the identity $b(g,h)$  by the commutativity of the diagram 
\begin{equation}\label{def: b}
\xymatrix{
\Id_{\cC} \ar[rr]^{b(g,h) } && \Id_{\cC}  \\
T(g)\circ T(h)   \ar[u]^{\chi_g\circ \chi_h}\ar[rr]^{\phi(g,h)}&& T(gh) \ar[u]^{\chi_{gh}}}
\end{equation}
where $\circ$ is the composition in $\Aut{\cC}$.

Hence,
\begin{enumerate}

\item the map $b:G\times G\to \Aut{\Id_\cC}$ defines a 2-cocycle in $Z^2(G,\operatorname{Aut}_\otimes(\operatorname{Id}_\cC))$ and its cohomology class does not depend on the choice of  the natural isomorphisms $\chi(\sigma)$,  $\sigma \in G$,
\item the $G$-action $T$ is trivializable if and only $$0=[b]\in H^2(G,\operatorname{Aut}_\otimes(\operatorname{Id}_\cC)),$$ 
\item  in case that  $0=[b]$ the set of  all trivializations is a non-empty torsor over $\operatorname{Hom}(G,\operatorname{Aut}_\otimes(\operatorname{Id}_\cC))$.
\end{enumerate}
\end{theorem}
\begin{proof}
(1) It follows from Lemma \ref{lemma: transport structure} that the $G$-action $T$ is equivalent to the $G$-action where $T'(g)=\Id_\cC$ as monoidal functor and $b(g,h):T'(g)\circ T'(h)=\Id_\cC\to T'(gh)=\Id_{\cC}$. Hence, the monoidal condition of $T'$ translate directly to commutativity of the diagram

\[
\xymatrix{
&\Id_{\cC}=T'(g)\circ T'(h)\circ T'(k) \ar[dl]_{b(g,h)\circ \Id_\cC} \ar[dr]^{\Id_\cC\circ b(h,k)}& \\ 
\Id_{\cC}=T'(gh)\circ T'(k) \ar[dr]_{b(gh,k)} &&\Id_{\cC}=T'(g)\circ T'(hk) \ar[dl]^{b(g,hk)}\\
&T'(ghk)=\Id_{\cC}&
}
\]that in equations translate to the 2-cocycle condition
\begin{align*}
b(gh,k)b(g,h)=b(g,hk)b(h,k) && \text{ for all } g,h,k\in G.    
\end{align*}
Now, if $\chi'_g: T(g)\to \Id_{\cC}$ is another family of monoidal natural isomorphisms, then $u_g:=\chi'_g\circ \chi_g^{-1}\in \Aut{\cC}$ and then using the naturality of $u_g$ we have that
\begin{align}\label{eq: b'}
b'(g,h)=u_{gh}\circ u_g^{-1}\circ u_h^{-1} \circ b(g,h) && \forall g,h\in G,
\end{align}hence the cohomology of $b$ does not depend on $\chi$.

(2) By definition, if $\eta$  is a trivialization of $T$ the associated 2-cocycle is trivial. Conversely, if a family of monoidal isomorphisms $\{\chi_g:T(g)\to \Id_\cC\}_{g\in G}$ defines a $b\in Z^2(G,\Aut{\Id_\cC})$ such that there is $u:G\to \Aut{\Id_\cC}$ such that $b(g,h)=u_{gh}^{-1} u_gu_h$ for all $g,h  \in G$ then the family $\{\eta_g:=u_g\chi_g|g\in G\}$  defines a trivialization of $G$.

(3) If $\{\eta_g: g\in G\}$ and $\{\eta_g': g\in G\}$ are trivialization then $\eta_g=\eta'u_g$ for a map $u:G\to \Aut{\Id_{\cC}}$. Hence by equation \eqref{eq: b'} we have that $u$ is a group homomorphism.
\end{proof}

\section{Graded braided monoidal categories as  crossed braided fusion categories with a trivialization}

A $G$-graded monoidal category is  a  monoidal category $\cC$ endowed with a decomposition $\C=\coprod_{g\in G} \cC_g$  (coproduct of categories) such that
 \begin{itemize}
 \item $\mathbf{1} \in \cC_e$,
 \item $\cC_g\otimes \cC_h \subset \cC_{gh}$ for all $g,h\in G$.
 \end{itemize}
If $\ku$ is a commutative ring and $\cC$ is a $\ku$-linear abelian category, the coproduct $\cC=\coprod_{g\in G} \C_g$ is taken in the category of $\ku$-linear abelian categories. 

\subsubsection{Braided $G$-crossed monoidal categories}

Let $T:\underline{G}\to \underline{\operatorname{Aut}_{\otimes}(\cC)}$ be an action of a group $G$ on $\cC$. Given $X, Y\in Ob(\cC)$ and $f:X\to Y$, we will denote by $g_*(X)$ and $g_*(f)$ the image of $X$ and $f$ under the functor $T(g)$.
\begin{definition}
Let $G$ be a group. A  $G$-crossed monoidal category is a monoidal category $\cC$ equipped
with the following structures:

\begin{enumerate}
    \item[(i)] an action of $G$ on $\cC$,
    \item[(ii)] a $G$-grading $\cC=\coprod_{g\in G}\cC_g$,
    \item[(iii)]  isomorphisms \begin{align*}
        c_{X,Y}:X\otimes Y\to g(Y)\otimes X, && g\in G, X\in \cC_g, Y\in \cC,
    \end{align*}natural in $X$ and $Y$. The isomorphisms $c_{X,Y}$ are called the \emph{$G$-braiding}.
\end{enumerate}This structures should satisfy the following conditions:
\begin{enumerate}
    \item[(a)] $g_*(\cC_h) \subseteq \cC_{ghg^{-1}}$, for all $g, h \in G$,
    \item[(b)] The diagrams
   \begin{equation*}\label{axiom 1 trenza}
\begin{tikzcd}
g_*(X\ot Z) \ar{dd}{\can} \ar{rrrr}{g_*(c_{X,Z})} &&&& g_*(h_*(Z)\ot X)\ar{dd}{\can}\\\\   
g_*(X)\ot g_*(Z) \ar{rrrr}{c_{g_*(X),g_*(Z)}} &&&& (ghg^{-1})_*g_*(Z)\ot g_*(X) 
\end{tikzcd}
\end{equation*}
commute  for all $X\in \cC_h, Z\in \cC, g,h\in G$.
\item[(c)]  The  diagrams

\begin{equation}\label{trenzas G 1}
\begin{tikzcd}
X \ot Y\ot Z \ar{rr}{c_{X,Y\ot Z}} \ar{d}{c_{X,Y}\ot\id_{Z}} &&   g_*(Y\ot Z)\ot X \ar{d}{\can} \\
g_*(Y)\ot X\ot Z \ar{rr}{\id_{g_*(Y)}\ot c_{X,Z}}&&  g_*(Y)\ot g_*(Z) \ot X
\end{tikzcd}
\end{equation}
commute for all $X\in \cC_g, Y,Z\in \cC$ and the diagrams 

\begin{equation}\label{trenzas G 2}
\begin{tikzcd}
X \ot Y\ot Z \ar{rr}{c_{X\ot Y, Z}} \ar{d}{\id_X\ot c_{Y,Z}} &&   (gh)_*(Z)\ot X\ot Y  \ar{d}{\can} \\
X\ot h_*(Z)\ot Y\ar{rr}{c_{X,h_*(Z)}\ot \id_Y}&&  g_*h_*(Z)\ot X \ot Y
\end{tikzcd}
\end{equation}
commute for all $X\in \cC_g, Y\in \cC_h, Z\in \cC, g,h\in G$.
\end{enumerate}
The isomorphisms $\can$ are the natural isomorphisms constructed using the natural isomorphisms of the action of $G$ on $\cC$. 
\end{definition}



A braided $G$-crossed monoidal category is strict if the $G$-action is strict. By \cite[Theorem 5.6]{GALINDO-Coherence}, every  $G$-crossed category is equivalent to a strict braided $G$-crossed monoidal category. Hence, in some proofs we will consider strict $G$-crossed braided categories without loss of generality.

\subsection{$A$-crossed braided tensor categories with trivialization}

Let $A$ be an abelian group. We define the 2-category of  $A$-crossed braided tensor categories with a \emph{trivializations} as follows: 

(1) objects are pairs $(\cC, \eta)$, where $\cC$ is a   braided $A$-crossed monoidal category  and $\eta$ is a trivializations of the $A$-action. 

(2) A 1-morphism from $(\cC, \eta )$ to $(\cC', \eta' )$ is a $A$-graded monoidal functor $(F,F_2):\cC \to \cC'$ such that the diagram 

\begin{equation}\label{diagram:condition-braided_trivialization}
\xymatrix@C-1em{
F(X_g\ot Y) \ar[dd]_{F_2(X_g,Y)} \ar[rrrr]^{F(c_{X_g,Y})} &&&& F(g_*(Y)\otimes X_g)\ar[dd]^{F_2(g_*(Y),X_g)}\\\\   
F(X_g)\ot F(Y) \ar[dd]_{c_{F(X_g),F(Y)}} &&&& F(g_*(Y))\ot F(X_g)) \ar[dd]^{F(\eta(g)_Y)\ot \id_{F(X_g)}}\\\\
g_*(F(Y))\ot F(X_g) \ar[rrrr]^{\eta'(g)_{F(Y)}\ot \id_{F(X_g)}  }  &&&& F(Y)\otimes F(X_g) }    
\end{equation}
commutes for all $X_g\in \cC_g, Y\in \cC$ and $g\in G$. 

(3) A 2-arrow is just a monoidal natural isomorphism.

\begin{theorem}
If $(\cC,c,\eta)$ is a  $A$-crossed braided monoidal category with a trivialization, the natural isomorphisms $c_{X_g,Y}^{(\eta)}:X_g\otimes Y\to Y\otimes X_g$ given by

\begin{equation}
\begin{tikzcd}
X_g\otimes Y \ar{rr}{c_{X_g,Y}^{(\eta)} } \ar{rd}{c_{X_g,Y}} && Y\otimes X_g\\
&  g_*(Y)\otimes X_g \ar{ru}{\eta(g)_Y\otimes \id_{X_g}} &
\end{tikzcd}
\end{equation}
define a braiding on $\cC$.

This assignment defines a biequivalence between the 2-category of  $A$-crossed braided fusion categories with a trivialization and the 2-category of  $A$-graded braided tensor categories.
\end{theorem}
\begin{proof}
The proof is a straightforward computation and even can be deduced directly from  \cite[Appendix 5, Proposition 2.3]{MR3195545}.
\end{proof}

\begin{example}\label{Ex:pointed}

Let $G$ be a group (non necessarily finite) and $\ku$ a field or a commutative ring. A (normalized)  3-cocycle $\omega \in Z^3(G, \ku^\times)$ is a map $\omega:G\times G\times G\to \ku^{\times}$ such that 
\begin{align*}
\omega(ab,c,d)\omega(a,b,cd)&=
\omega(a,b,c)\omega(a,bc,d)\omega(b,c,d), & \omega(a,1,b)=1,
\end{align*}
for all   $a,b,c,d\in G.$ 

Let us recall the description of the tensor category $\Vec_G^\omega$. The objects of $\Vec_G^\omega$ are $G$-graded $\ku$-modules $V=\bigoplus_{g\in G} V_{g}$. Morphisms are $\ku$-linear $G$-homogeneous maps. 
The tensor product of $V=\oplus_{g\in G}V_g$ and $W=\oplus_{g\in G}W_g$  is
$V\otimes_{\ku} W$ as $\ku$-module, with $G$-grading
\[(V\otimes W)_g=\bigoplus_{h\in G}V_h\otimes_{\ku} W_{h^{-1}g}.\]
For objects $V, W, Z \in \Vec_G^\omega$ the associativity constraint is defined by \begin{align*}
a_{V,W,Z}: (V\otimes W)\otimes Z&\to V\otimes (W\otimes Z)\\ 
(v_g\otimes w_h)\otimes z_k &\mapsto \omega(g,h,k) v_g\ot( w_h\otimes z_k)
\end{align*}
for all $g,h,k \in G, v_g\in V_g, w_h\in W_h, z_k\in Z_k$.  The unit objects is $\ku_e$, the $\ku$-module  $\ku$ graded only by the identity element $e\in G$.

Let $A$ be an abelian group and $\omega \in Z^3(A,\ku^\times)$ and define the maps
\begin{align*}
\mu(\sigma,\tau| \rho) &:= \frac{\omega(\tau, \sigma,\rho)}
{\omega(\tau, \rho, \sigma)\omega(\sigma, \tau, \rho)}\\
\gamma(\sigma|\tau,\rho) &:= \frac{\omega(\sigma, \tau, \rho) \omega(\rho,\sigma,\tau)}
{\omega(\sigma, \rho, \tau)},
\end{align*}
for all $\sigma$, $\tau$, $\rho\in A$. 

The tensor category $\Vec_A^\omega$ has a canonical braided $A$-crossed structure with $A$-action defined as follows: for each $g\in A$, the associated tensor functor is  $$g_*:=\Id,\  
\psi(g)_{a,b}= \gamma(g|a,b)\operatorname{id}_{\ku_{ab}}$$ and for each pair $g,h \in A$, the tensor natural isomorphism is
\begin{align*}
T_2(g,h)_{\ku_a}=\mu(g,h;a)\operatorname{id}_{\ku_a}, &&  a\in A.    
\end{align*}

The functor  $(g_*, \psi(g))$ is equivalent to the identity if and only if the cohomology class of $\gamma(g|-,-)\in Z^2(A,\ku^\times)$ is trivial.

For $\eta \in C^2(A,\ku^\times)$ define $$\delta_v(\eta)(a_1|a_2,a_3)=\eta(a_1,a_2)\eta(a_1,a_3)/\eta(a_1,a_2a_3)$$ and $$\delta_h(\eta)(a_1,a_2|a_3)=\eta(a_1,a_3)\eta(a_2,a_3)/\eta(a_1a_2|a_3).$$

Assume that $0=[\gamma(a|-,-)]\in H^2(A,\ku^\times)$
for all $a\in A$. Thus, there exists $\eta:A\times A\to \ku^\times$ such that $\delta_v(\eta)=\gamma$. The obstruction of Theorem \ref{obstruccion general} is given by
$$b(\eta)=\delta_h(\eta)\mu \in Z^2(A,\Hom(A,\ku^\times)),$$ since in this case $\Aut{\Id_{\Vec_A^\omega}}=\Hom(A,\ku^\times)$. As a conclusion, we obtain that a 3-cocycle $\omega\in Z^3(A,\ku^\times)$ admits a braiding if and only if the cohomology class of $b(\eta)$ vanishes, see \cite{MR3605649} for more details in this direction.

\end{example}

\section{$\Z/2$-bradings and braidings over Tambara-Yamagami categories}

Before compute the $\Z/2$-bradings and braidings over Tambara-Yamagami we will recall the concept of relative brading.
\subsection{Strongly graded central extensions}\label{section strongly graded central extensions}

For future computations, will recall the following notion.
\begin{definition}
Let $\cC$ be a monoidal category and $\cB\subset \cC$ be a monoidal subcategory. A  \emph{relative braiding}  consists of a natural family of isomorphisms \[\{c_{A,X}:A\ot X\to X\ot A\}_{A\in \cB, X\in \cC}\] such that the  diagrams 

\begin{equation}\label{H1}
    \begin{tikzcd}
&               A\ot (X\ot Y) \arrow[r,"c_{A,X\ot Y}"]&  (X\ot Y)\ot A \arrow[rd,"a_{X,Y,A}"] &\\
(A\ot X) \ot Y \arrow[ru, "a_{A,X,Y}"] \arrow[rd, " c_{A,X} \ot \id_Y"']&  &  & X \ot (Y\ot A)\\
&  (X\ot A)\ot Y \arrow[r,"a_{X,A,Y}"']& X\ot (A\ot Y) \arrow[ru," \id_X \ot c_{A,Y}"']
\end{tikzcd}
\end{equation}
and
\begin{equation}\label{H2}
\begin{tikzcd}
&               (A\ot B)\ot X \arrow[r,"c_{A\ot B,X}"]& X\ot (A\ot B) \arrow[rd,"a_{X,A,B}^{-1}"] &\\
A\ot (B \ot X) \arrow[ru, "a_{A,B,X}^{-1}"] \arrow[rd, "\id_A\ot c_{B,X}"']&  &  & (X\ot A)\ot B\\
& A\ot (X\ot B) \arrow[r,"a_{A,X,B}^{-1}"']& (A\ot X)\ot B \arrow[ru,"c_{A,X}\ot \id_Y"']
\end{tikzcd}
\end{equation}
commute for all $A, B \in \cB, X, Y\in \cC$.

\end{definition}

\begin{remark}
\begin{itemize}
    \item[(a)] A relative braiding is a central inclusion $\cB\to \mathcal{Z}(\cC)$ such that the composition with the forgetful functor $\mathcal{Z}(\cC)$ is the identity of $\cB$.
    
    \item[(b)] If $\cC$ is a faithful braided $G$-crossed category then the $G$-braiding $$c_{A_e,X_g}: A_e\otimes X_g\to X_g\otimes A_g, \quad \quad A_e\in \cC_e, X_g\in \cC,$$ defines a relative braiding.
\end{itemize}
\end{remark}

Let $\cB=\cC_e\subset \cC=\coprod_{g\in G}\cC_g$ a central $G$-extension. For each $g\in G$ we have monoidal functors

\begin{align}
\alpha^g:\cB\to \End_\cB(\cC_g), && \alpha^G_X(M_g)=X\ot M_g\\
\beta^g:\cB\to \End_\cB(\cC_g), && \beta^g_X(M_g)= M_g\ot X
\end{align}
with natural isomorphisms

\begin{align*}
     a_{Y,X,M_g}\circ c_{Y,X}^{-1}\ot\id_{M_g}\circ a_{X,Y,M_g}^{-1}:\alpha^g_X(Y\ot M_g) \to Y\ot \alpha^g_X(M_g)\\
    a_{Y,M_g,X}: \beta^g_X(Y\ot M_g) \to Y\ot \beta^g_X(M_g)
\end{align*}
for all $X,Y\in \cB, M_g\in \cC_g$.

\begin{definition}
A  $G$-graded extension $\cB\subset \cC$  will be
called a \emph{strongly graded} central $G$-extension if $\alpha^g$ and $\beta^g$ are equivalence of categories for each $g\in G$.
\end{definition}

If $\cB\subset \cC$ is a strongly graded central extension, then for every $g\in G$  there is a unique (up to equivalence) braided autoequivalence  $T(g):\cB\to \cC$ determined by the existence of a tensor equivalence 

\begin{equation}\label{eq: definition of action}
\alpha^g \circ T(g) \cong \beta^g.    
\end{equation}

Hence the monoidal functor $T(g)$ is naturally equivalent to the identity, if and only if $\alpha^g\cong \beta^g$  for all $g\in G$.
\subsection{Tambara-Yamagami fusion categories}
In this section, we collect some definitions and well-known facts about Tambara-Yamagami categories that we will need.

In \cite{TY} D.~Tambara and S.~Yamagami classified all
$\mathbb{Z}/2\mathbb{Z}$-graded
fusion categories  in which all but one of the
simple objects are invertible. 

\begin{definition}
Let $A$ be a finite abelian group. The \emph{Tambara-Yamagami fusion rules} are defined  over $A\cup \{m\}$ with product \begin{align*}
a\ot b =ab,\quad a\ot m = m,&& m \ot a=m,&& m \ot m
=\bigoplus_{a\in A}\, a,    
\end{align*}
for all $a, b\in A,$ and unit element $e\in A$.

Let $A$ be a finite abelian group, $\chi:A\times A\to \ku^\times$ a symmetric non-degenerate bicharacter and $\tau \in \ku^\times$ a square root  of
$|A|^{-1}$. The Tambara-Yamagami category $\TY(A,\chi, \tau)$ is the skeletal fusion category with Tambara-Yamagami fusion rules, strict unit object and non-identities associativity constraints 
\begin{align}
    \alpha_{a,m,b}=\chi(a,b)\id_m:m&\to m\\
    \alpha_{m,a,m}=\bigoplus_{b\in A}\chi(a,b)\id_b:\bigoplus_{b\in A}b &\to \bigoplus_{b\in A}b,\\
    \alpha_{m,m,m}=\big ( \tau \chi(a,b)^{-1}\id_m\big )_{a,b}:\bigoplus_{a \in A}m &\to \bigoplus_{b \in A}m
\end{align}
\end{definition}

\begin{remark}
The category $\TY(A,\chi, \tau)$ is rigid with 
    \begin{align*}
        a^*=a^{-1},&& coev_a= ev_a=\id_e, && a\in A,  
    \end{align*}
    and $m^*=m$,  $coev_m:e\to m\ot m$ the canonical injection and $ev_m=\tau^{-1}p: m\otimes m \to e$, where $p$ is the projection on $e$.

\end{remark}

Let $\operatorname{Aut}(A,\chi)$ be the group of automorphism of $A$ that respect $\chi$. Hence any $f\in \Aut{A,\chi}$ defines a \emph{strict} tensor auto-equivalence of $\TY(A,\chi, \tau)$ by 
\begin{align}\label{tensor autoequivalences TY}
F_f(a)=f(a), && F_f(m)=m,&&  a\in A.
\end{align}
Moreover, it was proved in \cite[Proposition 1]{Tambara-fiber} and \cite[Proposition 2.10]{MR2480712} that every tensor autoequivalence has the form $F_f$ for a unique $f\in \operatorname{Aut}(A,\chi)$.

\subsection{Relative braidings  for a Tambara-Yamagami fusion category}

Recall that if $A$ is an abelian group then a \emph{quadratic form} on $A$ with values in $\ku^\times$ is a function $q:A\to \ku^\times$ such that the symmetric function $w(a,b)=\frac{q(ab)}{q(a)q(b)}$ is a bicharacter and $q(a^{-1})=q(a)$ for all $a\in A$.

\begin{proposition}\label{Prop: TY braided crossed structures}
The relative braidings on $\Vec_A\subset \TY(A,\chi, \tau)$ are in correspondence with quadratic forms  $q:A\to \ku^\times$ such that
\begin{align}\label{eq: quadratic form TY}
\chi(a,b)=\frac{q(a)q(b)}{q(ab)}, && \forall a,b \in A.    
\end{align}
Moreover, 
\begin{itemize}
    \item[(a)] The relative braiding associated to a quadratic form $q:A\to \ku^\times$ satisfying \eqref{eq: quadratic form TY} is given by 
    \begin{align}
        c_{a,b}=\chi(a,b)\id_{ab}, && c_{a,m}=q(a)\id_m, && a,b\in A.
    \end{align}
    \item[(b)] The braided autoequivalence defined by \eqref{eq: definition of action} is the strict tensor automorphism
    \begin{align}
        T_1(a)=a^{-1}, && a\in A.
    \end{align}
    \item[(c)] Two relative braiding corresponding to quadratic forms $q$ and $q'$ are equivalent if and only if and only if there is a $f\in \operatorname{Aut}(A)$ such that $q'(f(a))=q(a)$ for all $a\in A$.
\end{itemize}

\end{proposition}
\begin{proof}
(a)\  \  For objects $X,Y,Z\in \TY(A,\chi, \tau)$ we will denote by $H_1(X,Y,Z)$ and $H_2(X,Y,Z)$ the hexagons \eqref{H1} and \eqref{H2} respectively.

Braiding on $\Vec_A$ are in corresponds with bicharacter $c:A\times A\to \ku^\times$ via $c_{a,b}=c(a,b)\id_{ab}$ for all $a, b \in A$. Let us denote by $q:A\to \ku^\times$ the function defined by  $c_{a,m}=q(a)\id_{m}$. 

The commutativity of the hexagon $H_1(a,m,b)$ is equal  to
\begin{align*}
  q(a)\chi(a,b)=c(a,b)q(a),
\end{align*}
that is $\chi(a,b)=c(a,b)$. The commutativity of the hexagon $H_2(a,b,m)$ is the equation 
\begin{equation}\label{eq:multi q}
q(ab)=q(a)\chi(a,b)^{-1}q(b).    
\end{equation}

The commutativity of the hexagon $H_1(a,m,m)$ is exactly the equation $\chi(a,a^{-1}b)=q(a)\chi(a,b)q(a)$, or equivalently $\chi(a,a)^{-1}=q(a)^2$. In presence of equation \eqref{eq:multi q}, the equation  $\chi(a,a)^{-1}=q(a)^2$ is equivalent to $q(a)=q(a^{-1})$.

Finally, the commutativity of diagram $H_1(a,b,m)$ is $\chi(b,a)q(a)=q(a)\chi(a,b)$ that follows from the symmetry of $\chi$.

(b) \ \ We will follow the notation and results from Section \ref{section strongly graded central extensions}. The simple objects of $\End_{\Vec_A}(\Vec_{m},\Vec_{m})$ are in bijective correspondence with elements in $\widehat{A}$ the group of characters of $A$. In fact, giving $\alpha \in A$, the tensor functor $F=\Id_{\Vec_m}$ with the natural isomorphism
\begin{align*}
F_2(a,m)=\alpha(a)\id_{a\ot m}, && a\in A,
\end{align*}
define a simple object in $\End_{\Vec_A}(\Vec_{m},\Vec_{m})$. Under this correspondence we have that
\begin{align*}
    \alpha_a=\chi(a,-)^{-1}=\chi(a^{-1},-), && \beta_a=\chi(-,a), && \forall a\in A.
\end{align*}Hence, if  $T^1(a)=a^{-1}$ for all $a\in A$ we have that $\alpha^1\circ T^1=\beta^1$.
 
(c)\ \ Let $q$ and $q'$ quadratic forms defining relative braiding for $\TY(A,\chi, \tau)$ and $f \in \operatorname{Aut}(A,\chi)$  such that $F_f:\TY(A,\chi, \tau)\to \TY(A,\chi, \tau)$ defined by \eqref{tensor autoequivalences TY} is an equivalence of central extensions, that is
 \[
 \begin{tikzcd}
 F_f(a\ot m) \arrow[r, "F_f(c_{a,m})"] \arrow[d, equal]& F_f(m\ot a)\arrow[d, equal]\\
 f(a)\ot m \arrow[r, "c_{f(a),m}"] & m\ot f(a)
 \end{tikzcd}
 \]Hence, $q(a)=q'(f(a))$ for all $a\in A$. Conversely, if there is $f\in \operatorname{Aut}(A)$ such that $q=q'\circ f$, then $\chi=\chi\circ f\times f$ and the tensor auto-equivalence $F_f$ defines an equivalence of central extensions.
\end{proof}

The following corollary explains the conditions founded in \cite{siehler2000braided} for the existence of braidings in a Tambara-Yamagami fusion category.

\begin{corollary}\label{coro: braiding implies elementary abelian 2-g}
If $\TY(A,\chi, \tau)$ admits a braiding then $A$ is an elementary abelian 2-group, that is, $a^2=e$ for all $a\in A$.
\end{corollary}
\begin{proof}
If  $\TY(A,\chi, \tau)$ admits a braiding, then by Proposition \ref{Prop: TY braided crossed structures} admits a relative braiding with associated $T^1=\Id_{\Vec_A}$, that is such that $a=a^{-1}$ for all $a\in A$. Hence $A$ is an elementary abelian 2-group. 
\end{proof}

\subsection{$\Z/2$-braidings for Tambarara-Yamagami fusion categories}

\begin{lemma}\label{Lemma: Z_2 actions pon TY}
Let $\TY(A,\chi, \tau)$ be a Tambara-Yamagami fusion category. Then,
\begin{enumerate}
    \item There is a unique non-trivial tensor natural isomorphism of  $\Id_{\TY(A,\chi, \tau)}$, namely $\gamma_a=\id_a$  for all $a\in A$ and $\gamma_m=-\id_m$.
\item Up to equivalence there are exactly two $\Z/2\Z$-actions with $T_1(a)=a^{-1}$. The first action is the strict action and the second one has monoidal natural isomorphism 
\begin{align*}
\gamma:T_1\circ T_1=\Id\to \Id, && \gamma_m=-\id_m, && \gamma_a=\id_a
\end{align*}
for all $a\in A.$ 
\end{enumerate}
\end{lemma}
\begin{proof}
The first item follows immediately from \cite[Proposition 3.9]{GELAKI20081053} and the second one follows from \cite[Theorem 5.5 (iii)]{Ga1}.
\end{proof}

\begin{theorem}\label{thm: Z/2 braiding TY}
The $\Z/2\Z$-braidings of $\TY(A,\chi, \tau)$ are in correspondence with  pairs $(q,\alpha)$, where $q:A\to \ku^\times$ is a quadratic form such that
\begin{align*}
\chi(a,b)=\frac{q(a)q(b)}{q(ab)}, && \forall a,b \in A.    
\end{align*}
and $\alpha\in\ku^{\times}$ such that $\alpha^2=\tau \Big (\sum_{a\in A}q(a) \Big)$. 

Moreover, given such a pair $(q,\alpha)$, the braided $\Z/2$-crossed structure is given as follows:

\begin{itemize}
    \item The $\Z/2$-action on $\TY(A,\chi, \tau)$ is strict and determined by \begin{align*}
        T(a)=a^{-1}, && T(m)=m, && a\in A.
    \end{align*}
    \item The $\Z/2$-braiding is given by 
    \begin{align}\label{eq: Z/2 brading}
        c_{a,b}=\chi(a,b)\id_{ab}, && c_{a,m}=c_{m,a}=q(a)\id_{m}, && c_{m,m}=\alpha\bigoplus_{a\in A} q(a)^{-1}\id_a
    \end{align}for all $a\in A$.
\end{itemize}
\end{theorem}
\begin{proof}
Let $\cC=\TY(A,\chi, \tau)$ with $\cC_0=\Vec_A$ and $\cC_1= \Vec_m$ and $q:A\to \ku^\times$ a quadratic form defining a relative braiding. 

It follows from Lemma \ref{Lemma: Z_2 actions pon TY} that there are only two possible $\Z/2\Z$-actions with $T_1(a)=a^{-1}$ for all $a\in A$. The only difference with the two actions is given by the automorphism $\gamma: T_1\circ T_1 \to \Id$, in the strict $\Z/2\Z$-action we have that $\gamma=\id$ and in the second one $\gamma_m=-\id_m$.

The diagrams \eqref{trenzas G 1} and \eqref{trenzas G 2} for  strict $g_*$'s are  written as

\begin{equation}\label{HH1}
    \begin{tikzcd}
&               X\ot (Y\ot Z) \arrow[r,"c_{X, Y\ot Z}"]&  (g(Y)\ot g(Z))\ot X \arrow[rd,"a_{Y,Z,X}"] & \\
(X\ot Y) \ot Z \arrow[ru, "a_{X,Y,Z}"] \arrow[rd, " c_{X,Y} \ot \id_Z"']&  &  & Y \ot (Z\ot X)\\
&  (g(Y)\ot X)\ot Z \arrow[r,"a_{Y,X,Z}"']& g(Y)\ot (X\ot Z) \arrow[ru,"\ot \id_\ot Y c_{X,Z}"']
\end{tikzcd}
\end{equation}where $X \in \cC_g$

\begin{equation}\label{HH2}
    \begin{tikzcd}
&               (X\ot Y)\ot Z \arrow[r,"c_{X\ot Y,Z}"]& gh(Z)\ot (X\ot Y) \arrow[r,"{\gamma^{g,h}_Z\ot 1}"] & g(h(Z))\ot (X\ot Y)\arrow[d,"a^{-1}"]\\
X\ot (Y \ot Z) \arrow[ru, "a{-1}"] \arrow[rd, "\id_X\ot c_{Y,Z}"']&  &  & (Z\ot X)\ot Y\\
& X\ot (Z\ot Y) \arrow[r,"a^{-1}"']& (X\ot Z)\ot Y \arrow[ru,"c_{X,Z}\ot 1"']
\end{tikzcd}
\end{equation}
where $X\in\cC_g, Y\in\cC_h$. We will denote as $HH_1(X,Y,Z)$ and $HH_2(X,Y,Z)$ the diagrams \eqref{HH1} and \eqref{HH2} respectively.

Let us denote by $w, s :A\to \ku^\times$ the functions defined by 
\begin{align*}
    c_{m,a}=w(a)\id_m:m\ot a &\to -a\ot m,\\
    c_{m,m}=\oplus_{\in A}s(a)\id_a:m \ot m =\bigoplus_{a\in A}a &\to m\ot m=\bigoplus_{a\in A}a .
\end{align*}

The commutativity of diagram $HH_1(m,a,m)$ is the equation 
\begin{equation}\label{eq HH_2(m,a,m)}
\chi(a,b)s(b)=w(a)s(ab), \quad \forall b\in A,
\end{equation}
and commutativity of diagram $HH_2(m,a,m)$ is $\chi(a,b)^{-1}s(b)=q(a)s(ba^{-1})$ for all $b\in B$. Using that $q(a)=q(a^{-1})$ we have that $HH_2(m,a,m)$ commutes if and only if \[\chi(a,b)s(b)=q(a)s(ab), \quad \forall b \in A. \]
Hence the commutativity of diagrams $HH_1(m,a,m)$ and $HH_2(m,a,m)$ for all $a\in A$ is equivalent to $w(a)=q(a)$ for all $a\in A$. Moreover, we have from \eqref{eq HH_2(m,a,m)} that 
\begin{equation}\label{eq: s in terms of q}
    s(a)=s(e)q(a)^{-1}, \quad \forall a\in A. 
\end{equation} Hence, $c_{m,m}=s(e)\bigoplus_{a\in A}q(a)^{-1}\id_a$.

The commutativity of diagram $HH_1(m,m,m)$ is equal to 
\[
s(a)s(c)\chi(a,c)^{-1}=\tau\sum_{b\in A}q(b)\chi(b,a^{-1}c), \quad \forall a, c \in A.
\]Equivalently, using that $s(a)=s(a^{-1})$ the commutativity of $HH_1(m,m,m)$ is equal to the equation
\[
s(a)s(c)\chi(a,c)=\tau\sum_{b\in A}q(b)\chi(b,ac), \quad \forall a, c \in A.
\]
Taking $a=c=e$ we have that \[s(e)^2=\tau  \big (\sum_{b\in A}q(b)\big ). \]
The commutativity of the diagram $HH_2(m,m,m)$ for the strict $\Z/2\Z$-action is equal to 
\[
s(a)s(c)\chi(a,c)=\tau\sum_{b\in A}q(b)\chi(b,ac), \quad \forall a, c \in A.
\]and for the non-strict action is equal to
\[
s(a)s(c)\chi(a,c)=-\tau\sum_{b\in A}q(b)\chi(b,ac), \quad \forall a, c \in A.
\]Hence, only the strict action admits a $\Z/2\Z$-braiding.
\end{proof}
The following result is a reinterpretation of the main result of \cite{siehler2000braided}.
\begin{corollary}\cite[Theorem 1.2]{siehler2000braided}
\begin{itemize}
    \item[(a)] A Tambara-Yamagami fusion  category  $\TY(A,\chi, \tau)$ admits a braiding if and only if $A$ is an elementary abelian 2-group.
    \item[(b)] If $A$ is an elementary abelian 2-group, there is a correspondence between braiding and pairs $(q,\alpha)$ where $q$ is a quadratic form such that
\begin{align*}
\chi(a,b)=\frac{q(a)q(b)}{q(ab)}, && \forall a,b \in A.    
\end{align*} and $\alpha \in \ku^\times$ such that $\alpha^2=\tau(\sum_{a\in A}q(a))$. 
The braiding associated to a pair $(q,\alpha)$ is given by the formulas in \eqref{eq: Z/2 brading}.
\item[(c)] Two braidings associated to $(q,\alpha)$ and $(q',\alpha')$ are equivalent if and  only if there is $f\in \operatorname{Aut}(A)$ such that $q'(f(a))=q(a)$ for all $a\in A$ and $\alpha=\alpha'$.
\end{itemize}
\end{corollary}
\begin{proof}
If follows from Corollary \eqref{coro: braiding implies elementary abelian 2-g} that $A$ must be elementary abelian if a braidings for  $\TY(A,\chi, \tau)$ exist. In this case, the $\Z/2$-action is trivial. Hence a $\Z/2$-braiding is exactly a braiding. Now, the first two items of the corollary follow from Theorem \ref{thm: Z/2 braiding TY}.

Let $(q,\alpha)$ and $(q',\alpha')$ pairs  defining two braidings for $\TY(A,\chi, \tau)$ and $f \in \operatorname{Aut}(A,\chi)$  such that $F_f:\TY(A,\chi, \tau)\to \TY(A,\chi, \tau)$ defined by \eqref{tensor autoequivalences TY} is an equivalence of braided categories. Then we should have $F_f(c_{a,m})=c_{f(a),m}$ for all $a\in A$ and $F_f(c_{m,m})=c_{m,m}$, that is, $q(a)=q'(f(a))$ for all $a\in A$ and $\alpha=\alpha'$ respectively.
\end{proof}

\begin{remark}
In \cite[Section 4B]{GNN},  the authors studied braided $\Z/2$-crossed structures over $\mathcal{Z}_{\Z/2}(\TY(A,\chi, \tau))$ (the equivariant Drinfeld center). Using an equivariant central inclusion of $\TY(A,\chi, \tau)$ in $\mathcal{Z}_{\Z/2}(\TY(A,\chi, \tau))$, it should be possible to describe braided $\Z/2$-crossed structures of Tambara-Yqmagami categories. However, following this approach, our formulas do not agree with the formulas in \cite{GNN}, particularly the $\Z/2$-action consider in \cite{GNN} does not agree with the action of Theorem \ref{thm: Z/2 braiding TY}.
\end{remark}

\subsection{Ribbon $\Z/2$-crossed structures}

Since ribbons of braided $G$-crossed fusion categories play an essential part in the construction of homotopical TFT's, we finish the paper with the computation of ribbon for the $\Z/2$-braiding constructed in the previous section.
 
Let $A$ be an abelian group and $\cB=\coprod_{a\in A}\cB_a$ a strict braided $A$-crossed monoidal category. A \emph{twist} is a natural isomorphism 
\begin{align*}
\theta_{X}:X\to a_*(X), && X\in \cB_a,     
\end{align*}such that 
\begin{itemize}
    \item[(Tw1)] $\theta_{\one}=\id_{\one}$,
    \item[(Tw2)] $b_*\theta_{X}=\theta_{b_*X}$
    \item[(Tw3)]
    \[
    \xymatrix{
    X_a\ot Y_b \ar[d]^{c_{X_a,Y_b}}\ar[rrr]^{\theta_{X_a\ot Y_b}} &&& (ab)_*X_a\ot (ab)_*Y_b \\
     a_*(Y_b)\ot X_a \ar[rrr]^{c_{a_*(Y_b),X_a}} &&& b_*(X_a)\ot a_*(Y_b) \ar[u]_{b_*\theta_{X_a}\ot a_*\theta_{Y_b}}
    }
    \] for all $X_a\in \cB_a, Y_b\in \cB_b$, $a,b \in A$.
\end{itemize}
A twist that satisfies the condition \[\theta_{X_a^*}=(a^{-1})_*(\theta^*_{X_a})\]for all $X_a\in  \cB_a$, $a\in A$ is called a \emph{ribbon}.

\begin{proposition}
Let $\TY(A,\chi, \tau)$ be a Tambara-Yamagami with a $\Z/2$-braiding defined by a pair $(q,\alpha)$. Then $\TY(A,\chi, \tau)$ admits exactly two $\Z/2$-ribbon structures given by
\begin{align}
    \theta_a=q(a)^{-2}, && \theta_m=\beta, && a\in A,
\end{align}where $\beta^{-2}=\tau\Big (\sum_{a\in A}q(a) \Big )$.
\end{proposition}
\begin{proof}
The condition (Tw3) for $X=a\in A, Y=m$ is equivalent to $\theta_{-a}=q(a)^{-2}$. Again, condition (Tw3) for $X=m\in A, Y=m$ is equivalent to $\theta_m^2\alpha^2=1$, hence $\theta_m^{-2}=\tau\Big (\sum_{a\in A}q(a) \Big ).$
\end{proof}
\begin{example}[Modular structures on Ising fusion rules]
As a concrete example, we provide a classification of modular Ising categories, giving an alternative proof of some of the results in \cite[Appendix B]{DGNO}.

The Ising fusion rules corresponds to Tambara-Yamagami fusion rules with $A=\Z/2=\{\one, \psi\}$, that is, the simple objects are $\{\one, \psi, m\}$ with fusion rules \begin{align*}
    m^2 =\one+\psi, && \psi m= m\psi \psi =m, && \psi^2= \one. 
\end{align*}
The group $\Z/2$ has only one  non-degenerate symmetric bicharacter determined by $\chi(\psi,\psi)=-1$. Then, there are up to equivalence two fusion categories with Ising fusion rules, namely
\[\TY(\Z/2,\chi,\pm \frac{1}{\sqrt{2}}).\]
There are two quadratic forms given by \[q_{\pm i}(\psi)=\pm i\]  with associated symmetric bicharacter $\chi$, then each $\TY(\Z/2,\chi,\pm \frac{1}{\sqrt{2}})$ admits four different braiding corresponding to the pairs $(q_{ i}, \pm e^{\frac{2\pi i}{8}})$ and $(q_{ -i}, \pm e^{\frac{3\pi i}{8}})$. Finally, each braided fusion category admits two ribbon structures.
\end{example}


\end{document}